\documentclass[a4paper,12pt]{amsart}
\usepackage{amssymb,enumerate}

\newcommand{\FF}{{\mathbb{F}}}
\newcommand{\NN}{{\mathbb{N}}}
\newcommand{\QQ}{{\mathbb{Q}}}
\newcommand{\RR}{{\mathbb{R}}}
\newcommand{\ZZ}{{\mathbb{Z}}}

\newcommand{\fp}{{\mathfrak{p}}}
\newcommand{\fS}{{\mathfrak{S}}}

\newcommand{\cH}{{\mathcal{H}}}
\newcommand{\cK}{{\mathcal{K}}}
\newcommand{\cN}{{\mathcal{N}}}
\newcommand{\cO}{{\mathcal{O}}}
\newcommand{\cP}{{\mathcal{P}}}

\newcommand{\Aut}{{\operatorname{Aut}}}
\newcommand{\AutO}{{\operatorname{Aut}_\cO}}

\newcommand{\rk}{{\operatorname{rk}}}
\newcommand{\Cl}{{\operatorname{Cl}}}

\newcommand{\GL}{{\operatorname{GL}}}
\newcommand{\Sp}{{\operatorname{Sp}}}

\newcommand{\SPk}[3]{\Sp_{#1}^{(#2)}(\cO/\fp^{#3})}
\newcommand{\Op}[1]{\cO/\fp^{#1}}
\newcommand{\Z}[1]{\ZZ/p^{#1}\ZZ}
\newcommand{\cGp}{{\mathcal{G}_p}}

\def\pmod#1{~({\rm mod}~#1)}

\let\vhi=\varphi

\newtheorem{thm}{Theorem}[section]
\newtheorem{lem}[thm]{Lemma}
\newtheorem{conj}[thm]{Conjecture}
\newtheorem{cor}[thm]{Corollary}
\newtheorem{prop}[thm]{Proposition}

\theoremstyle{definition}
\newtheorem{exmp}[thm]{Example}
\newtheorem{defn}[thm]{Definition}

\theoremstyle{remark}
\newtheorem{rem}[thm]{Remark}

\begin{document}

\title[A class group heuristic based on 1-eigenspaces]{A class group heuristic
  based on the\\ distribution of 1-eigenspaces in matrix groups}

\date{\today}

\author{Michael Adam \and Gunter Malle}
\address{FB Mathematik, TU Kaiserslautern, Postfach 3049,
         67653 Kaisers\-lautern, Germany.}
\email{adam@mathematik.uni-kl.de}
\email{malle@mathematik.uni-kl.de}

\thanks{The authors gratefully acknowledge financial support by the Deutsche
  Forschungsgemeinschaft}

\keywords{Cohen--Lenstra heuristic, class groups, symplectic groups, 1-eigenspaces}

\subjclass[2010]{Primary 11R29, 15B33; Secondary  20P05}

\begin{abstract}
We propose a modification to the Cohen--Lenstra prediction for the distribution
of class groups of number fields, which should also apply when the base field
contains non-trivial roots of unity. The underlying heuristic derives from the
distribution of 1-eigenspaces in certain generalized symplectic groups
over finite rings. The motivation for that heuristic comes from the function
field case. We also give explicit formulas for the new predictions in several
important cases. These are in close accordance with known data.
\end{abstract}

\maketitle


\section{Introduction}
The class group of a number field $K$ is defined as the quotient
$\Cl(K):= I_{K}/P_{K}$ of the group of fractional ideals $I_{K}$ by the
subgroup of principal ideals $P_{K}$ in $\cO_{K}$, the ring of integers of $K$.
Despite its importance, not much is known about the behavior of these objects.
For instance it is still an open question whether there exist
infinitely many number fields with trivial class group.
In the early 1980's H. Cohen and H.W. Lenstra \cite{CL83} proposed a
heuristic principle and presented some conjectural statements to describe the
distribution of $p$-parts of class groups of imaginary quadratic number fields.
This concept was extended by Cohen and J. Martinet \cite{CM90} to arbitrary
extensions of a fixed number field $K_0$. These results make predictions on
how often a given finite abelian $p$-group should appear as the $p$-part of
the class
group in a specified set of number fields. Cohen and Martinet formulated their
conjectures for all but a finite set of prime numbers for which the
behavior of $\Cl(K)_p$ is already controlled by genus theory. Only very few
instances of these conjectures have been proved (see \cite{Bh05,FK07} for
important recent progress).
\par
Seven years ago it was noticed by the second author \cite{Ma08,Ma10} that
in the situation when $p$th roots of unity are contained in the base field
$K_0$ the postulated probabilities by Cohen and Martinet do not match with
computational data. In particular this is always the case for $p=2$. In the
absence of theoretical arguments, on the basis of his computational data
the second author came up with a conjectural statement \cite[Conj. 2.1]{Ma10}
describing the behavior of $p$-parts of class groups in the presence of
$p$th roots of unity in $K_0$.
\par
Motivated by the works of J. Achter \cite{Ach06, Ach08}, who considered the
analogous problem on the function field side, we develop a method which can
be seen as a theoretical justification for the conjectures of Cohen and
Martinet and at the same time for the above-mentioned conjecture in
\cite{Ma10}. Here, the main objects are 1-eigenspaces of elements in what we
call the $m$-th symplectic groups $\Sp_{2n}^{(m)}(R)$ over certain finite rings
$R$ (see Definition~\ref{def:isymp}). The symplectic groups come up due to the
existence of the Tate pairing on class groups. The limit for
$n\rightarrow\infty$ of these eigenspace distributions should then give the
right predictions for class group distributions over number fields.
\par
We use the results proved by the first author \cite{Adam} to compute
distributions in these $m$-th symplectic groups (see Theorems~\ref{thm:level0}
and~\ref{thm:level1}) which allows us to make conjectural predictions
(see Conjecture~\ref{conj:Main}) about the behavior of $p$-parts of class
groups of number fields. For the case when the base field does not contain
$p$th roots of unity, these specialize to the original Cohen--Lenstra--Martinet
predictions.

\section{Cohen--Lenstra heuristic and roots of unity}
In this section we recall the heuristic principle introduced by Cohen and
Lenstra to predict the distribution of $p$-parts
of class groups of imaginary quadratic number fields and the generalization to
arbitrary number fields proposed by Cohen and Martinet. However, our
focus lies on a situation where these predictions seem to fail.

\subsection{The Cohen--Lenstra heuristic}
Cohen and Lenstra \cite{CL83} considered computational data on class groups
of imaginary quadratic number fields. They noticed that for example $\Z{2}$
appears much more often as class group than $\Z{}\times\Z{}$. They "explained"
this by the fact that the latter group has many more automorphisms than the
former. A straightforward conclusion from this observation was to equip every
group with the following weight.

\begin{defn}
 The \emph{CL-weight} of a finite group $G$ is
 $\omega(G):=\dfrac{1}{|\Aut(G)|}$.
\end{defn}

For integers $q,r\ge1$ we set $(q)_r:=\displaystyle\prod_{i=1}^{r} (1-q^{-i})$
and  $(q)_{\infty}:=\displaystyle\prod_{i=1}^{\infty} (1-q^{-i})$.

\begin{prop}
 For a prime~$p$, let $\cGp$ denote the set of all isomorphism classes of
 finite abelian $p$-groups. Then
 $$\sum_{G\in\cGp} \omega(G) = (p)_\infty^{-1}<\infty.$$
\end{prop}

See \cite[Ex.~5.10]{CL83} for a proof. This immediately implies:

\begin{cor}   \label{cor:ver}
 The function given by
 $$P_{CL}:\cGp\longrightarrow[0,1],\ G\mapsto\dfrac{(p)_\infty}{|\Aut(G)|},
 $$
 defines a probability distribution on $\cGp$.
\end{cor}

It turned out that $P_{CL}$ plays the role of a kind of natural distribution
on the set of finite abelian $p$-groups and occurs also in many other
contexts. See \cite{Le10} for an overview of this issue.

Write $\cK$ for the set of all imaginary quadratic extensions of $\QQ$ inside
a fixed algebraic closure of $\QQ$. Then given a finite abelian $p$-group $G$
we set
$$\cN(G):=\lim_{x\rightarrow\infty}\dfrac{|\{K\in\cK : |d_K|\leq x,\
  \Cl(K)_p\cong G\}|}{|\{K\in\cK : |d_K|\leq x\}|}$$
(if it exists), where $d_K$ denotes the field discriminant of the number
field $K$.

Using this notation we can formulate the first conjecture stated by Cohen and
Lenstra which makes predictions about the behavior of $p$-parts of class
groups of imaginary quadratic number fields.

\begin{conj}[Cohen and Lenstra]   \label{conj:CL}
 Let $p$ be an odd prime. Then for a finite abelian $p$-group $G$ the limit
 $\cN(G)$ exists and is equal to $P_{CL}(G)$.
\end{conj}

Cohen and Martinet \cite{CM90} continued the ideas presented above and
formulated an analoguous conjecture for arbitrary extensions
of a number field $K_0$. In order to present their result we need to
introduce the following set up.

\begin{defn}
 We call a triple $\Sigma:=(H,K_0,\sigma)$ a \emph{situation}, where
 \begin{enumerate}[(1)]
  \item $H\leq \fS_n$ is a transitive permutation group of degree $n\geq 2$,
  \item $K_0$ is a number field, and
  \item $\sigma$ is a signature which may occur as signature of a degree $n$
   extension $K/K_0$ with Galois group (of the Galois closure) permutation
   isomorphic to $G$.
 \end{enumerate}
 For a situation $\Sigma=(H,K_0,\sigma)$ we let $\cK(\Sigma)$ denote the set
 of number fields $K/K_0$ (inside a fixed algebraic closure) as described
 in~(3).
\end{defn}

To a given situation $\Sigma$ one can attach a non-negative rational number
$u(\Sigma)$, the \textit{unit rank} and a ring $\cO(\Sigma)$. We give the
definition of these terms in a very important special case, the general case
is explained in \cite[Chap.~6]{CM90}.

We write $\chi$ for the permutation character attached to the embedding
$H\le\fS_n$. It contains the trivial character $1_H$ exactly once, and we let
$\chi_1:=\chi-1_H$. Let $\QQ[H]$ denote the rational group ring of $H$. We
make the following two assumptions:
\begin{itemize}
 \item[(1)] $\chi_1$ is the character of an irreducible (but not necessarily
  absolutely irreducible) $\QQ[H]$-module,
 \item[(2)] any absolutely irreducible constituent $\vhi$ of $\chi_1$ has
  Schur index~1, that is, $\vhi$ is the character of a representations of $H$
  over the field of values of $\vhi$.
\end{itemize}
Then $\cO(\Sigma)$ is defined to be the ring of integers of the field of
values of any absolutely irreducible constituent $\vhi$ of $\chi_1$. (This is
an abelian, hence normal extension of $\QQ$, and thus independent of the
choice of constituent $\vhi$ by~(1) above.)
\par
Now for $K\in\cK(\Sigma)$ let $L$ denote the Galois closure of $K/K_0$.
Denote by $E_L$ the group of units of the ring of integers of $L$. Then the
action of $H$ makes $E_L\otimes_\ZZ\QQ$ into a $\QQ[H]$-module, whose character
we denote by $\chi_E$. (It can be computed explicitly in terms of the
signature $\sigma$ of $L/K_0$ by the theorem of Herbrandt, see
\cite[Thm.~6.7]{CM90}.) Then
$$u(\Sigma):=\langle\chi_E,\vhi\rangle$$
(see \cite[p.~63]{CM90}), the scalar product of the character $\chi_E$ with an
absolutely irreducible constituent $\vhi$ of $\chi_1$. Since $\chi_E$ is
rational, this does not depend on the choice of $\vhi$ (nor does it depend on
the choice of $K\in\cK(\Sigma)$).

Given a situation $\Sigma=(H,K_0,\sigma)$ and a finite $\fp$-torsion
$\cO=\cO(\Sigma)$-module $G$, where $\fp\trianglelefteq\cO$ is a prime ideal,
we set
$$\cN(\Sigma, G):=\lim_{x\rightarrow\infty}\dfrac{|\{K\in\cK(\Sigma) :
  d_{K/K_0}\leq x,\ \Cl(K/K_0)_\fp\cong G\}|}
  {|\{K\in\cK(\Sigma) : d_{K/K_0}\leq x\}|}$$
(if it exists), where $d_{K/K_0}$ denotes the norm of the discriminant of
$K/K_0$ and $\Cl(K/K_0)$ is the relative class group of $K/K_0$ (the kernel
in the class group of $K$ of the norm map from $K$ to $K_0$). With this we
can present the conjecture of Cohen and Martinet predicting the distribution
of $\fp$-parts of relative class groups of number fields over~$K_0$.

\begin{conj}[Cohen and Martinet]   \label{conj:CM}
 Let $G$ be a finite $\fp$-tosion $\cO$-module. Then $\cN(\Sigma, G)$
 exists and is given by
 $$\dfrac{(q)_\infty}{(q)_u}\cdot\dfrac{1}{|G|^u|\Aut_\cO(G)|},$$
 where $u:= u(\Sigma)$, $q:=|\cO/\fp|$, and $\AutO(G)$ denotes the group
 of $\cO$-automor\-phisms of $G$ 
\end{conj}

For this consult \cite[Chap.~6]{CM90} and \cite[Ex.~5.9]{CL83}.
\par
In the following years it was noticed by many people that this conjecture
cannot hold for all primes which were originally allowed by Cohen and Martinet.
So we need to lay the focus a bit more on the set of bad primes. The reason
why Cohen and Martinet \cite{CM90} excluded the primes that divided the
extension degree $n=(K/K_0)$ is that by genus theory these primes are
indeed bad, meaning that the conjecture cannot be true for such primes. 
A few years later and after more
computations Cohen and Martinet \cite{CM94} were forced to enlarge the set of
bad primes by those which divide the order of the common Galois group $H$ of
the situation $\Sigma$. For the bad behavior of these primes one can find
theoretical arguments in the manner of genus theory, too. Much later it was
noticed by the second author \cite{Ma08,Ma10} on the basis of extensive
numerical support that the presence of $p$th roots of unity in the base field
$K_0$ does play a role for the distribution of $p$-parts of class groups.
In an attempt to explain this deviation, one is led to consider
the analogous situation on the side of function fields.

\subsection{The function field case}
Three years after the publication of the celebrated paper of Cohen and Lenstra
\cite{CL83} the analoguous case on the function field side was considered.
Given two distinct prime numbers $l$ and $p$, E.~Friedman and L.C.~Washington
\cite{FW89} linked the distribution of $p$-parts of the divisor class groups
of degree 0 of quadratic extensions of $\FF_l(t)$ to equi-distributed sequences
of matrices over finite fields.
Achter \cite{Ach06,Ach08} took up this suggestion and proved that, to put it
simply, the distribution of class groups of quadratic function fields and the
distribution of elements in symplectic similitude groups are the same. More
precisely, let $\cH_{g}(\FF_l)$ denote the set of monic separable
polynomials of degree $g$ over the finite field $\FF_l$ and let $C_{g,f}$ be
the hyperelliptic curve of genus $g$ defined by $f\in\cH_{2g+1}(\FF_l)$. With
this for a finite abelian $p$-group $G$ we define the proportion
$$\beta_p(g,l,G):=\dfrac{|\{f\in\cH_{2g+1}(\FF_l) :
   \Cl^0(C_{g,f})_p\cong G\}|}{|\cH_{2g+1}(\FF_l)|},$$
where $\Cl^0(C_{g,f})_p$ is the Sylow $p$-subgroup of the Jacobian of
$C_{g,f}$. Among other things Achter showed for a finite elementary abelian
$p$-group $G$ that
$$\lim_{l\rightarrow\infty} \beta_p(g,l,G) =
  \dfrac{|\{h\in\Sp_{2g}(\FF_p) : \ker(h-\mathbf{1}_{2g})\cong G\}|}
  {|\Sp_{2g}(\FF_p)|}.$$
Later Achter considered symplectic similitude groups over finite rings and
established a correspondence between eigenspace distributions in such groups
and the distribution of the Jacobian of hyperelliptic curves. (Consult
\cite[Thm.~3.1]{Ach08} for details.) However, Achter did not compute the
distributions in the symplectic similitude groups explicitly. This step has
been done by J.S.~Ellenberg, A.~Venkatesh and C.~Westerland \cite{EVW} and
one consequence of their work is that the distribution of the $p$-parts of
divisor class groups of degree $0$ of quadratic function fields over $\FF_l$
with $p{\not|}(l-1)$ (which corresponds on the
number field side to the case where $p$th roots of unity are not contained
in the base field) matches the distributions predicted by Cohen--Lenstra and
Friedman--Washington. The case where $p|(l-1)$ was treated by D.~Garton,
a student of Ellenberg, in his recent thesis \cite{Ga12}, where he presents
a distribution for this situation \cite[Thm.~7.2.3]{Ga12}. Following
\cite[\S3]{Ma10} the philosophy should now be that the characteristic~0 number
field case can be obtained as the limit for the genus $g\rightarrow\infty$ of
the characteristic~$l$ function field cases.

\section{Distribution of 1-eigenspaces in matrix groups}
The ideas and results from the function field case yield the motivation for
a more thorough investigation of eigenspaces in suitable finite matrix groups.
Since all finite abelian $p$-groups should occur as such eigenspaces, and
should carry an alternating form coming from the Tate pairing, we are led to
consider symplectic groups over finite rings that are not integral domains.
We introduce these groups and recall the crucial results shown by the first
author in \cite{Adam}.

\begin{defn}   \label{def:isymp}
 Let $\cO$ be the ring of integers of a number field and let
 $\fp\trianglelefteq\cO$ be a non-zero prime ideal. Given natural numbers
 $n$ and $m\le f$ we define the \emph{$m$-th symplectic group} over
 the ring $\Op{f}$ as
 $$\Sp_{2n}^{(m)}(\Op{f}):=\{h\in\GL_{2n}(\Op{f})\mid
    h^tJ_{n}h\equiv J_{n} \pmod{\fp^{m}}\},$$
 where $\GL_{2n}(\Op{f})$ is the general linear group and
 $J_{n}:=\begin{pmatrix}  0& \mathbf{1}_{n} \\  -\mathbf{1}_n& 0\end{pmatrix}
  \in\GL_{2n}(\Op{f})$.
\end{defn}

\begin{rem}
 From the definition we obtain the following descending chain of groups:
 $$\GL_{2n}(\Op{f})=\Sp_{2n}^{(0)}(\Op{f})\supseteq\Sp_{2n}^{(1)}(\Op{f})
   \supseteq\dots\supseteq\Sp_{2n}^{(f)}(\Op{f})=\Sp_{2n}(\Op{f}),$$
 where $\Sp_{2n}(\Op{f})$ denotes the usual symplectic group over $\Op{f}$.
 Thus, the $m$-th symplectic groups in a sense `interpolate' between the
 general linear and the symplectic group over the non-integral domain $\Op{f}$.
\end{rem}

\begin{prop}
 Let $q:=|\cO/\fp|$ and $f\in\NN$. Then:
 \begin{enumerate}[\rm(a)]
  \item $|\Sp_{2n}^{(0)}(\Op{f})|=|\GL_{2n}(\Op{f})|
    =q^{4n^2(f-1)}\cdot|\GL_{2n}(\FF_{q})|$.
  \item $|\Sp_{2n}^{(f)}(\Op{f})|=|\Sp_{2n}(\Op{f})|
    =q^{(2n^2+n)(f-1)}\cdot|\Sp_{2n}(\FF_{q})|$.
  \item $|\Sp_{2n}^{(m)}(\Op{f})|=q^{4n^2(f-m)}\cdot|\Sp_{2n}(\Op{m})|$
   for $1\leq m\leq f-1$.
 \end{enumerate}
\end{prop}

\begin{proof}
See \cite[Prop.~2.7]{Adam}.
\end{proof}

We are interested in the following limit proportion of elements with a given
1-eigenspace
$$P_{m,q}(G):=
  \lim_{n\to\infty}\dfrac{|\{g\in \SPk{2n}{m}{f} :
     \ker(g-\mathbf{1}_{2n})\cong G\}|}{|\SPk{2n}{m}{f}|},$$
where $G$ denotes a finite $\fp$-torsion $\cO$-module. Note that this is
trivially a probability measure on the set of finite $\fp$-torsion
$\cO$-modules. For $m\le 2$ this distribution was computed explicitely in
\cite[Thms.~3.8 and~3.23]{Adam}, respectively \cite[Rem.~4.27]{Diss}:

\begin{thm}   \label{thm:distr}
 Let $G$ be a finite $\fp$-torsion $\cO$-module annihilated by $\fp^{f-1}$.
 Then:
 \begin{enumerate}[\rm(a)]
  \item $P_{0,q}(G)=\dfrac{(q)_\infty}{|\AutO(G)|}$,
  \item $P_{1,q}(G)=\dfrac{(q)_\infty}{(q^2)_\infty}
        \cdot \dfrac{(q)_r\cdot q^{\binom{r}{2}}}{|\AutO(G)|}$ with
   $r=\rk_\fp(G)$,
  \item $P_{2,q}(G)=\dfrac{(q)_\infty}{(q^2)_\infty}
        \cdot \dfrac{(q)_{r-s}(q)_s\cdot q^{\binom{r}{2}+\binom{s}{2}}}
        {(q^2)_t\,|\AutO(G)|}$ with $r=\rk_\fp(G)$, $s=\rk_{\fp^2}(G)$,
        $t=\lfloor\frac{r-s}{2}\rfloor$,
 \end{enumerate}
 where $q:=|\cO/\fp|$.
\end{thm}

In particular one sees that the limit does not depend on $\cO$ and $\fp$,
but only on $q$, and is also independent of $f$, as long as $G$ is
annihilated by $\fp^{f-1}$. The value of $|\AutO(G)|$ is computed
explicitly in \cite[Thm.~2.11]{CM90}.

\begin{rem}   \label{general set}
We have no closed formulas for $P_{m,q}(G)$ for $m\ge3$ and
general $G$, but for fixed $m$ and $G$ it is possible to calculate
$P_{m,q}(G)$ explicitly, as indicated in \cite{Adam}.
\end{rem}

\section{$u$-Probabilities}

The concept of $u$-probability was originally introduced by Cohen and Lenstra 
(see \cite[Chap.~5]{CL83}). Here we present a modified version used by various
other authors (e.g.~\cite{Le10}). Throughout, $\cO$ denotes the ring of
integers of an algebraic number field, $\fp\trianglelefteq\cO$ is
a non-zero prime ideal, and $q:=|\cO/\fp|$. We let $\cP$ denote the set of
isomorphism classes of finite $\fp$-torsion $\cO$-modules.

\begin{defn}\label{uprob}
 Given a probability distribution $P$ on $\cP$ and a natural number $u$ we
 define the \emph{$u$-probability distribution} with respect to $P$ by the
 following recursion formula
 $$P^{(u)}:\cP\longrightarrow\RR,\quad G\mapsto P^{(u)}(G):=
   \sum_{H\in\cP}\sum_{y\in H \atop H/\langle y\rangle\cong G}
      \dfrac{P^{(u-1)}(H)}{|H|},$$
 where $P^{(0)}(G):= P(G)$ for $G\in\cP$. Here, $\langle y\rangle$ denotes
 the $\cO$-submodule generated by $y$. We call $P^{(u)}(G)$ the
 \emph{$u$-probability} of $G$ (with respect to $P$).
\end{defn}

\begin{rem}
Note that $P^{(u)}$ is in fact a probability distribution on $\cP$, since
$$\sum_{G\in\cP}P^{(1)}(G)=\sum_{H\in\cP}\sum_{G\in\cP}
  \dfrac{|\{y\in H \mid H/\langle y\rangle\cong G\}|}{|H|}\cdot P(H)
  =\sum_{H\in\cP} P(H)=1$$
for any probability distribution $P$ on $\cP$, and $P^{(u+1)}=(P^{(u)})^{(1)}$.
\end{rem}

We now compute the $u$-probabilities for the distributions $P_{i,q}$ given in
Theorem~\ref{thm:distr}. For $G\in\cP$ a finite $\fp$-torsion $\cO$-module
of rank~$r$ and $k\ge0$ we set
$$w_k(G):=\begin{cases}
   \dfrac{(q)_k}{(q)_{k-r}|\AutO(G)|}& \text{ if } k\ge r,\\
   \qquad 0& \text{ else,}
\end{cases}$$
and $w_\infty(G):=|\AutO(G)|^{-1}=\lim_{k\rightarrow\infty}w_k(G)$.

With this we recall the following crucial result from \cite[Thm.~3.5]{CL83}:

\begin{prop}   \label{prop:CLprop}
 Let $Z,G\in\cP$. Then for all $0\le k\le\infty$ we have
 $$\sum_{H\in\cP} w_k(H)\dfrac{|\{H_1\leq H : H_1\cong Z,\ H/H_1\cong G\}|}
  {|\AutO(H)|}=w_k(Z)w_k(G).$$
\end{prop}

\begin{thm}   \label{thm:level0}
 Let $G$ be a finite $\fp$-torsion $\cO$-module, and $q:=|\cO/\fp|$. Then
 $$P_{0,q}^{(u)}(G)
   =\dfrac{(q)_\infty}{(q)_u}\cdot\dfrac{1}{|G|^u|\AutO(G)|}$$
 for all integers $u\geq 0$.
\end{thm}

\begin{proof}
The induction base $u=0$ is given by Theorem~\ref{thm:distr}, so now let
$u\ge1$. By Proposition~\ref{prop:CLprop} with $k=\infty$ we have 
$$\sum_{H\in\cP} \dfrac{|\{H_1\leq H : H_1\cong Z,\ H/H_1\cong G\}|}
  {|\AutO(H)|}=\dfrac{1}{|\AutO(Z)||\AutO(G)|}$$
for any $Z\in\cP$. With $Z=\cO/\fp^n\cO$ this gives
$$\sum_{H\in\cP \atop |H|=q^n|G|} \dfrac{|\{y\in H : |\langle y\rangle|=q^n,
  \  H/\langle y\rangle\cong G\}|}{|\AutO(H)|}=\dfrac{1}{|\AutO(G)|}.$$
Multiplying this equation by $(q^n|G|)^{-u}$ and summing over all $n\in\NN$
we obtain
$$\sum_{n\geq 0}\sum_{H\in\cP \atop|H|=q^n|G|}\dfrac{|\{y\in H :
    |\langle y\rangle| = q^n,\ H/\langle y\rangle\cong G\}|}{|H|^u|\AutO(H)|}
  = \sum_{n\geq 0}\dfrac{1}{q^{un}|G|^u|\AutO(G)|}$$
which by induction is equivalent to
$$\dfrac{(q)_{u-1}}{(q)_\infty}\sum_{H\in\cP}
  \sum_{y\in H \atop H/\langle y\rangle\cong G}\dfrac{P_{0,q}^{(u-1)}(H)}{|H|}
   =\dfrac{1}{|G|^u|\AutO(G)|}\sum_{n\geq 0}\dfrac{1}{q^{un}}
   = \dfrac{1}{|G|^u|\AutO(G)|}.$$
\end{proof}

Next, we determine the $u$-probabilities for the distribution given by the
first symplectic groups. For this we show first the following result.

\begin{lem}   \label{lem:m=1}
 For $G\in\cP$ of $\fp$-rank $r$ and all $u\in\NN$ we have
 $$\sum_{H\in\cP\atop\rk_\fp(H)=r}\dfrac{|\{y\in H : H/\langle y\rangle
   \cong G\}|}{|H|^u |\AutO(H)|}
   =\dfrac{q^{r+u}-1}{q^r(q^u-1)}\cdot\dfrac{1}{|G|^u|\AutO(G)|}.$$
\end{lem}

\begin{proof}
Let $u\in\NN$. Splitting up the sum in question according to the order of
$y$ we get
$$\sum_{H\in\cP\atop\rk_\fp(H)=r}\dfrac{|\{y\in H\mid H/\langle y\rangle\cong
  G\}|}{|H|^u |\AutO(H)|}
   =\sum_{n\ge0}\sum_{H\in\cP\atop\rk_\fp(H)=r}\dfrac{|\{y\in H :
  |\langle y\rangle|=q^n,\,H/\langle y\rangle\cong G\}|}{|H|^u |\AutO(H)|}.
$$
Writing $Z_n:=\cO/\fp^n\cO$, the inner sum equals
$$\begin{aligned}
  \sum_{H\in\cP\atop\rk_\fp(H)=r}|\AutO(Z_n)|&\cdot
    \frac{|\{H_1\le H : H_1\cong Z_n,\,H/\langle H_1\rangle\cong G\}|}
    {|H|^u |\AutO(H)|} \\
  & =\frac{|\AutO(Z_n)|}{(q)_r|G|^uq^{nu}}\sum_{H\in\cP} w_r(H)\,
    |\{H_1\le H : H_1\cong Z_n,\,H/\langle H_1\rangle\cong  G\}|\\
  & = \frac{|\AutO(Z_n)|}{(q)_r|G|^uq^{nu}}\,w_r(Z_n)w_r(G) \qquad
      \text{by Proposition~\ref{prop:CLprop}}.
\end{aligned}$$
Note that the middle sum may be extended over all $H$, since $w_r(H)=0$ if
$\rk_\fp(H)>r$.
Now $w_r(Z_0)=1$ and $w_r(Z_n)=(q)_r/(q)_{r-1}|\AutO(Z_n)|^{-1}$ when
$n>0$, so the left hand side in the assertion becomes
$$\begin{aligned}
  \frac{w_r(G)}{(q)_r|G|^u}\left(1+\sum_{n\ge1}\frac{w_r(Z_n)|\AutO(Z_n)|}
      {q^{nu}}\right) 
  & = \frac{w_r(G)}{(q)_r|G|^u}\left(1+\frac{(q)_r}{(q)_{r-1}}
      \sum_{n\ge1} \frac{1}{q^{nu}}\right)\\
  & = \frac{1}{|G|^u|\AutO(G)|}\Big(1+(1-q^{-r})\frac{1}{q^u-1}\Big)\\
  & = \frac{q^{r+u}-1}{q^r(q^u-1)}\cdot\frac{1}{|G|^u|\AutO(G)|}\\
\end{aligned}$$
as claimed.
\end{proof}

\begin{thm}   \label{thm:level1}
 Let $G$ be a finite $\fp$-torsion $\cO$-module of rank~$r$, and
 $q:=|\cO/\fp|$. Then
 $$P_{1,q}^{(u)}(G)=\dfrac{(q^2)_u(q)_\infty}{(q)_u(q^2)_\infty}
   \cdot\dfrac{(q)_{r+u} q^{\binom{r}{2}}}{(q)_u|G|^u |\AutO(G)|}.$$
\end{thm}

\begin{proof}
Let $G\in\cP$ of $\fp$-rank $r$. The case $u=0$ holds by
Theorem~\ref{thm:distr}(b). So by induction and Definition~\ref{uprob} we need
to compute
$$\dfrac{(q^2)_{u-1}(q)_\infty}{(q)_{u-1}(q^2)_\infty}
   \Big( (q)_{r+u-1}\,q^{\binom{r}{2}}\,X(r)
       + (q)_{r+u}\,q^{\binom{r+1}{2}}X(r+1)\Big)$$
where
$$X(s):=
  \sum_{H\in\cP\atop\rk_\fp(H)=s}\sum_{y\in H\atop H/\langle y\rangle\cong G}
    \dfrac{1}{(q)_{u-1}|H|^u|\AutO(H)|}\qquad\text{ for }s\in\{r,r+1\}.$$
With
$$Y:= \dfrac{1}{(q)_u|G|^u|\AutO(G)|}$$
Lemma~\ref{lem:m=1} states that
$X(r)=\dfrac{q^{r+u}-1}{q^{r+u}}\cdot Y$, and Theorem~\ref{thm:level0} gives
$Y=X(r)+X(r+1)$, so
$$X(r+1)=Y-X(r)=\Big(1-\frac{q^{r+u}-1}{q^{r+u}}\Big)Y=\frac{1}{q^{r+u}}\,Y.$$
Then the left hand side of the assertion becomes
$$\begin{aligned}
  \dfrac{(q^2)_{u-1}(q)_\infty\,q^{\binom{r}{2}}}{(q)_{u-1}(q^2)_\infty}
   &\Big( (q)_{r+u-1}\,\dfrac{q^{r+u}-1}{q^{r+u}}
       + (q)_{r+u}\,q^r \frac{1}{q^{r+u}}\Big)\, Y\\
  & = \dfrac{(q^2)_{u-1}(q)_\infty (q)_{r+u}\,q^{\binom{r}{2}}}
   {(q)_{u-1}(q^2)_\infty q^u}(q^u+ 1)\, Y
 = \dfrac{(q^2)_u(q)_\infty}{(q)_u(q^2)_\infty}\cdot
   \dfrac{(q)_{r+u}\,q^{\binom{r}{2}}}{(q)_u|G|^u|\AutO(G)|}
\end{aligned}$$
as claimed.
\end{proof}

\section{Distribution of class groups of number fields}
We can now present our conjecture about the distribution of $p$-parts of
class groups using the results from the last section. Here we restrict
ourselves to situations $\Sigma$ such that $\cO(\Sigma)=\ZZ$.

\begin{conj}   \label{conj:Main}
 Let $p$ be a prime, $\Sigma=(H,K_0,\sigma)$ be a situation with
 $\gcd(p,|H|)=1$ such that  $\cO(\Sigma)=\ZZ$, and $K_0$ be a number field
 containing the $p^{m}$th but not the $p^{m+1}$th roots of unity. Then a given
 finite abelian $p$-group $G$ occurs as the $p$-part of a relative class
 group $\Cl(K/K_0)$ for $K\in\cK(\Sigma)$ with probability $P_{m,p}^{(u)}(G)$,
 where $u=u(\Sigma)$.
\end{conj}

Let us consider some special cases of this conjecture in which we have derived
explicit formulas.

\begin{exmp}
In the case $m=0$, which should correspond to situations were no non-trivial
$p$th roots of unity are contained in the base field,
Theorem~\ref{thm:level0} yields that the probability in
Conjecture~\ref{conj:Main} that a finite abelian $p$-group $G$ occurs as
the $p$-part of a class group is given by
$$P_{0,p}^{(u)}(G)=\dfrac{(p)_\infty}{(p)_u}\cdot\dfrac{1}{|G|^u|\Aut(G)|}.$$
This is exactly the probability predicted by Cohen, Lenstra and Martinet (see
Conjecture~\ref{conj:CM}).
\end{exmp}

\begin{exmp}
Next consider the case $m=1$, which should apply when $p$th but no higher
roots of unity are present. Then by Theorem~\ref{thm:level1} the
distribution in Conjecture~\ref{conj:Main} is given by
$$P_{1,p}^{(u)}(G)=\dfrac{(p^2)_u(p)_\infty}{(p)_u(p^2)_\infty}
  \cdot\dfrac{(p)_{r+u} p^{\binom{r}{2}}}{(p)_u|G|^u |\Aut(G)|}.$$
This distribution is exactly the one proposed by the second author in
\cite[Conj.~2.1]{Ma10} for the case when the base field contains $p$th but
no higher roots of unity. This was derived from, and is in very close
accordance with, huge amounts of computational data in a number of situations,
see \cite{Ma10} for details, but had no heuristic underpinning. Our approach
via eigenspaces in $m$th symplectic groups gives a theoretical explanation for
the above formula. As shown in \cite[Prop.~2.2]{Ma10}, the probability for a
class group to have $p$-rank $r$ would then equal
$$\frac{(p^2)_u(p)_\infty}{(p)_u(p^2)_\infty}\cdot
   \frac{1}{p^{r(r+2u+1)/2}\,(p)_r}.$$
\end{exmp}

\begin{exmp}
Finally, assume that $m=2$, which should apply if the $p^2$rd but no higher
roots of unity are present in the base field. For $u=0$, the value
of $P_{2,p}^{(0)}(G)$ is given in Theorem~\ref{thm:distr}(c). The general
formulae for $u\ge1$ seem to get quite messy, therefore we only give some
example values. When $u=1$ we find
$$P_{2,p}^{(1)}(\ZZ/p^k\ZZ)=\dfrac{(p^2)_u(p)_\infty}{(p)_u(p^2)_\infty}
  \cdot\dfrac{(p)_{r+u}\, p^{\binom{r}{2}}}{(p)_u|\ZZ/p^k\ZZ|^u\, |\Aut(\ZZ/p^k\ZZ)|}$$
(for this calculation we have to sum over groups of types
$\ZZ/p^\alpha\ZZ\times\ZZ/p^\beta\ZZ$ with $\alpha\ge k\ge\beta$), while for
the smallest non-cyclic $p$-group the result is
$$P_{2,p}^{(1)}(\ZZ/p\ZZ\times \ZZ/p\ZZ)
  =\dfrac{(p)_\infty}{(p^2)_\infty}\cdot\dfrac{p^3+p^2-1}{p^7(p-1)}.$$
From the previous result we obtain
$$P_{2,p}^{(2)}(1)=\dfrac{(p^4)_1(p)_\infty}{(p)_1(p^2)_\infty}$$
for the trivial group at $u=2$.
\end{exmp}

\begin{rem}
As mentioned in Remark~\ref{general set} we are not able to give a closed
formula for $P_{m,p}^{(u)}(G)$ for arbitrary $m$. However, for small $m$ and
$u$ and small groups $G$ it should be possible to calculate $P_{m,p}^{(u)}(G)$
by following the ideas above.
\end{rem}

One might expect that the case $\cO\neq\ZZ$ can be treated with similar
methods as those presented above. The main problem seems to be to find a
suitable adaptation of the concept of $u$-probability to the case of
$\fp$-torsion modules. The obvious approach does not seem to yield results
which are in agreement with computational data from \cite{Ma08,Ma10}.



\begin{thebibliography}{13}

\bibitem{Ach06}
\textsc{J. D. Achter}, The distribution of class groups of function
  fields. J. Pure and Appl. Algebra {\bf204} (2006), 316--333.

\bibitem{Ach08}
\bysame, Results of Cohen--Lenstra type for quadratic function fields. In:
  \emph{Computational arithmetic geometry}, Contemp. Math., vol. 463,
  Amer. Math. Soc., Providence, RI, 2008, 1--7.

\bibitem{Diss}
\textsc{M. Adam}, On the distribution of eigenspaces in classical groups
  over finite rings and the Cohen--Lenstra heuristic. Dissertation,
  TU Kaiserslautern, 2014.

\bibitem{Adam}
\bysame, On the distribution of eigenspaces in classical groups
  over finite rings. Linear Algebra Appl. {\bf443} (2014), 50--65.

\bibitem{Bh05}
\textsc{M. Bhargava}, The density of discriminants of quartic rings and
  fields.  Ann. of Math. (2) {\bf162} (2005),  1031--1063. 

\bibitem{CL83}
\textsc{H. Cohen, H. W. Lenstra Jr.}, Heuristics on class groups of
  number fields. In: \emph{Number theory, Noordwijkerhout 1983}
  (Noordwijkerhout, 1983). Springer, Berlin,  1984, 33--62.

\bibitem{CM90}
\textsc{H. Cohen, J. Martinet}, \'Etude heuristique des groupes de
  classes des corps de nombres. J. reine angew. Math. {\bf404} (1990), 39--76.

\bibitem{CM94}
\bysame, Heuristics on class groups: some good primes are not too good.
  Math. Comp. {\bf63} (1994), 329--334.

\bibitem{EVW}
\textsc{J. S. Ellenberg, A. Venkatesh, C. Westerland},
  \emph{Homological stability for Hurwitz spaces and the Cohen--Lenstra
  conjecture over function fields.} arXiv:0912.0325 (2009).

\bibitem{FK07}
\textsc{E. Fouvry, J. Kl\"uners}, On the 4-rank of class groups of quadratic
  number fields.  Invent. Math. {\bf167} (2007),  455--513. 

\bibitem{FW89}
\textsc{E. Friedman, L. C. Washington}, On the distribution of
  divisor class groups of curves over a finite field. In: \emph{Th\'eorie des
  nombres} (Quebec, PQ, 1987), de Gruyter, Berlin, 1989, 227--239.

\bibitem{Ga12}
\textsc{D. Garton}, \emph{Random matrices and Cohen--Lenstra statistics for
  global fields with roots of unity.} Ph.D. thesis, UW Madison, 2012.

\bibitem{Le10}
\textsc{J. Lengler}, The Cohen--Lenstra heuristic: methodology and results.
  J. Algebra {\bf323} (2010), 2960--2976.

\bibitem{Ma08}
\textsc{G. Malle}, Cohen--Lenstra heuristic and roots of unity. J. Number
  Theory {\bf128} (2008), 2823--2835.

\bibitem{Ma10}
\bysame, On the distribution of class groups of number fields. Experiment.
  Math. {\bf19} (2010), 465--474.

\end{thebibliography}
\end{document}